\documentclass[reqno,11pt]{amsart}
\usepackage[foot]{amsaddr}
\usepackage{bbold}
\usepackage{charter}
\usepackage[margin=1in]{geometry}
\usepackage[colorlinks=true,linktoc=all,linkcolor=blue,citecolor=blue]{hyperref}

\theoremstyle{plain}
\newtheorem{theorem}{Theorem}[section]
\newtheorem*{theorem*}{Theorem}
\newtheorem{proposition}{Proposition}[theorem]
\newtheorem{lemma}[theorem]{Lemma}
\newtheorem{corollary}[theorem]{Corollary}
\theoremstyle{definition}

\newtheorem*{definition*}{Definition}
\newtheorem{remark}[theorem]{Remark}

\numberwithin{equation}{section}
\everymath{\displaystyle}
\allowdisplaybreaks

\newcommand{\C}{\mathbb{C}}
\newcommand{\D}{\mathbb{D}}
\newcommand{\N}{\mathbb{N}}
\newcommand{\ip}[2]{\left\langle #1,#2 \right\rangle}

\title{Multiplication Operators on $S^2(\D)$}
\author{Robert F.~Allen\textsuperscript{1}, Katherine C.~Heller\textsuperscript{2}, and Matthew A.~Pons\textsuperscript{2}}
\address{\textsuperscript{1}Department of Mathematics, University of Wisconsin-La Crosse}
\address{\textsuperscript{2}Department of Mathematics, North Central College}
\email{rallen@uwlax.edu, kheller@noctrl.edu, mapons@noctrl.edu}

\date{}


\begin{document}

\begin{abstract}
In this paper, we study the multiplication operators on \(S^2\), the space of analytic functions on the open unit disk \(\D\) whose first derivative is in \(H^2\). Specifically, we characterize the bounded and the compact multiplication operators, establish estimates on the operator norm, and determine the spectrum. Finally, we prove that the isometric multiplication operators are precisely those induced by a constant function of modulus one.
\end{abstract}

\maketitle

\section{Introduction}
Let $\mathcal{X}$ be a Banach space of analytic functions on the open unit disk $\D$ in $\C$.  Let $\psi$ be analytic on $\D$ and $\varphi$ an analytic self-map of $\D$.  We define, for $f \in \mathcal{X}$,
$$\begin{aligned}
C_\varphi(f) &= f\circ\varphi,\\
M_\psi(f) &= \psi f,\\
W_{\psi,\varphi}(f) &= \psi(f\circ\varphi),
\end{aligned}$$ the composition, multiplication, and weighted composition operators, respectively, on $\mathcal{X}$.  The composition operator has a long history, for which the reader is referred to \cite{Shapiro:93} and \cite{CowenMacCluer:95}.

The multiplication operator is not as well studied of an operator.  This is surprising since the study of multipliers on spaces can aid in the understanding of spaces of analytic functions, and can be a key tool in the study of weighted composition operators.  There are many classical spaces on which to study such functions.  These spaces include the weighted Hardy (see \cite{SharmaKomal:12}), Bergman (see \cite{Axler:82},\cite{GuoHuang:11}), Dirichlet (see \cite{Stegenga:80}), and Bloch (see \cite{Arazy:82}, \cite{BrownShields:91}, \cite{AllenColonna:09}) spaces.  In addition to the classical spaces, the study of multiplication operators extends between various spaces of analytic functions (see \cite{BonetDomanskiLindstrom:99}, \cite{Feldman:99} and \cite{LiuYu:2012}).  The main goal in the study of multiplication operators on these spaces is to link the properties of the operator $M_\psi$ with the properties of the symbol $\psi$.

In this paper, we study the multiplication operator on the space $S^2$, defined as the set of functions analytic on $\D$ whose derivative is a function in the Hardy Hilbert space $H^2$.  This is a specific instance in the family of spaces $S^p$ for $1 \leq p < \infty$.  The study of composition operators on $S^p$ began with Roan \cite{Roan:78}.  Subsequently, MacCluer characterized boundedness and compactness in terms of Carleson measures \cite{MacCluer:87}.  In her Ph.D. dissertation \cite{Heller:10}, the second author studied composition operators on $S^2$ in terms of the symbol.

Contreras and Hern\'{a}ndez-D\'{i}az in \cite{ContrerasHernandez-Diaz:04} studied the weighted composition operators between $S^p$ and $S^q$.  The boundedness and compactness of these operators were characterized in terms of other weighted composition operators between Hardy spaces.  In this paper, we wish to characterize boundedness and compactness of the multiplication operator in terms of its symbol.  In addition, we also study other properties of the operator, including the spectrum, commutant, and isometries.

\subsection{Organization of the paper}
In Section \ref{section:preliminaries}, we collect useful properties of functions in $S^2$ and $H^2$.  In Section \ref{section:boundedness}, we characterize the boundedness of a class of weighted composition operator on $S^2$.  In addition, we establish bounds on the operator norm.  From this, we obtain a characterization of the bounded multiplication operators on $S^2$ in terms of its symbol.

In Section \ref{section:spectrum}, we determine the spectrum of the bounded multiplication operators on $S^2$.  With this, in Section \ref{section:compact} we show the only compact multiplication operator on $S^2$ is the operator induced by the zero function.  

Finally, in Section \ref{section:isometries}, we characterize the isometric multiplication operators on $S^2$ as those induced by constant functions of modulus one.  We use this characterization to study isometric zero-divisors on $S^2$.

\section{Preliminaries}\label{section:preliminaries}
Let $\D$ denote the open unit disk in $\C$, and $H(\D)$ the space of analytic functions on $\D$.  In the field of operator theory, the classical spaces on $\D$ include the Hardy space, standard weighted Bergman spaces, and the Dirichlet space.  The Hardy space is defined by
$$H^2(\D) = \left\{f \text{ in } H(\D) : \|f\|_{H^2}^2 = \sup_{0 < r < 1} \int_0^{2\pi}|f(re^{i\theta})|^2\; \frac{d\theta}{2\pi} < \infty\right\},$$ where $d\theta$ is Lebesgue arc-length measure on the unit circle. For $\beta > -1$, the standard weighted Bergman space is defined by
$$A_\beta^2(\D) = \left\{f \text{ in } H(\D) : \|f\|_{A_\beta^2}^2 = \int_{\D} |f(z)|^2(1-|z|^2)^\beta\;dA<\infty\right\},$$
where $dA$ is Lebesgue area measure normalized so $A(\D) = 1$.  The Dirichlet space is defined by
$$\mathcal{D}(\D) = \left\{f \text{ in } H(\D) : \|f\|_{\mathcal{D}}^2 = |f(0)|^2 + \int_{\D}|f'(z)|^2\;dA<\infty\right\}.$$

A \textit{reproducing kernel Hilbert space} $\mathcal{H}$ with inner product $\ip{\cdot}{\cdot}_{\mathcal{H}}$ has the property  that for each $w \in \D$, there exists a unique function $K_w \in \mathcal{H}$, called the \textit{point-evaluation kernel}, such that $$f(w) = \ip{f}{K_w}_{\mathcal{H}}.$$  The Hardy space, standard weighted Bergman spaces, and the Dirichlet space are all reproducing kernel Hilbert spaces with kernels:
$$\begin{aligned}
H^2:\; &K_w(z) = \frac{1}{1-\overline{w}z},\\
A_\beta^2:\; &K_w(z) = \frac{1}{(1-\overline{w}z)^{\beta+2}},\\
\mathcal{D}:\; &K_w(z) = 1+\log\frac{1}{1-\overline{w}z},
\end{aligned}$$ where $\log z$ denotes the principle branch of the logarithm.

A \textit{functional Hilbert space} is a Hilbert space $\mathcal{H}$ whose elements are complex-valued functions on a set $\Omega$, with the usual addition of functions and multiplication by scalars, and such that evaluation at each point of $\Omega$ is a bounded linear functional on $\mathcal{H}$ and there is no point in $\Omega$ at which all functions of $\mathcal{H}$ vanish.  The Hardy space, standard weighted Bergman spaces, and the Dirichlet space are all functional Banach spaces with:
$$\begin{aligned}
\|K_w\|_{H^2}^2 &= \frac{1}{1-|w|^2};\\
\|K_w\|_{A_\beta^2}^2 &= \frac{1}{(1-|w|^2)^{\beta+2}};\\
\|K_w\|_{\mathcal{D}}^2 &= 1+\log\frac{1}{1-|w|^2}.
\end{aligned}$$

To see the natural relationship between these spaces, it is often convenient to define the spaces and norms in terms of series representations.  We have the alternative definitions of the three spaces:
$$\begin{aligned}
H^2(\D) &= \left\{f(z) = \sum_{n=0}^\infty a_nz^n \text{ in } H(\D) : \sum_{n=0}^\infty |a_n|^2 < \infty\right\},\\
A_\beta^2(\D) &= \left\{f(z) = \sum_{n=0}^\infty a_nz^n \text{ in } H(\D) : |a_0|^2 + \sum_{n=1}^\infty \frac{|a_n|^2}{n^{\beta+1}} < \infty\right\},\\
\mathcal{D}(\D) &= \left\{f(z) = \sum_{n=0}^\infty a_nz^n \text{ in } H(\D) : |a_0|^2 + \sum_{n=1}^\infty n|a_n|^2 < \infty\right\}.
\end{aligned}$$
From this, we see that $\mathcal{D} \subset H^2 \subset A^2$.  However, there seem to be many more spaces that can extend this containment.  One such space, which has received a great deal of attention of late, is the space $S^2$, defined by
$$\begin{aligned}
S^2(\D) &= \left\{f \in H(\D) : \|f\|_{S^2}^2 = |f(0)|^2 + \|f'\|_{H^2}^2 < \infty\right\}\\
&= \left\{f(z) = \sum_{n=0}^\infty a_nz^n \text{ in } H(\D) : \|f\|_{S^2}^2 = |a_0|^2 + \sum_{n=1}^\infty n^2|a_n|^2 < \infty\right\}.
\end{aligned}$$  Note that the integral and series norms are actually equal, not just equivalent.  We see that $S^2$ fits with the classical spaces since $S^2 \subset \mathcal{D} \subset H^2 \subset A^2$.  Unlike the other spaces, there is no ``nice" closed form for the reproducing kernels in $S^2$.

\noindent The following results are collected here for use in later sections of the paper.

\begin{theorem}\label{properties} For the space $S^2$:
\begin{enumerate}
\item[(a)] every function in $S^2$ extends continuously to the boundary of $\D$.
\item[(b)] every function in $S^2$ is bounded on $\D$.
\item[(c)] $S^2$ contains the polynomials.
\item[(d)] evaluation at each point in $\D$ is a bounded linear functional.
\item[(e)] $S^2$ is a functional Hilbert space.
\end{enumerate}
\end{theorem}

\begin{proof}
Properties (b),(c), and (d) can be found in Chapter 4 of \cite{CowenMacCluer:95}.  The fact that $S^2$ is a functional Hilbert space follows immediately from (c) and (d).  We will provide a proof of property (a), which shows $S^2$ to be a boundary-regular space.  Consider $f\in S^2$ with power series representation $f(z)=\sum_{n=0}^{\infty}a_nz^n$ and let $f_N$ denote the $N^{th}$-partial sum of this series.  Assuming $M \geq N$, it follows that $f$ is uniformly Cauchy on $\overline{\D}$ from H\"{o}lder's inequality and the estimate
$$\begin{aligned}\left|f_N(z)-f_M(z)\right|&\leq \sum_{n=N+1}^M|a_n|
&\leq\left(\sum_{n=N+1}^{M}|a_n|^2n^2\right)^{1/2}\left(\sum_{n=N+1}^{M}\frac{1}{n^2}\right)^{1/2}.\end{aligned}$$ It follows immediately that $f$ extends continuously to $\partial \D$.
\end{proof}

\begin{proposition}\label{sup-norm-inequality} Let $f \in S^2$.  Then $\|f\|_\infty \leq \displaystyle\frac{\pi\sqrt{3}}{3}\|f\|_{S^2}.$
\end{proposition}

\begin{proof}
Let $f(z) = \displaystyle\sum_{n=0}^\infty a_n z^n$ for $z \in \D$.  For $x,y \geq 0$, recall the following inequality
\begin{equation}\label{pons-inequality}
(x+y)^2 \leq 2(x^2+y^2).
\end{equation}
Then from H\"{o}lders's inequality, we have
$$\begin{aligned}
|f(z)|^2 &\leq \left[|a_0| + \sum_{n=1}^\infty |a_n|\right]^2\\
&\leq \left[|a_0| + \left(\sum_{n=1}^\infty n^2|a_n|^2\right)^{1/2}\left(\sum_{n=1}^\infty \frac{1}{n^2}\right)^{1/2}\right]^2\\
&\leq \left[|a_0|\left(\sum_{n=1}^\infty \frac{1}{n^2}\right)^{1/2} + \left(\sum_{n=1}^\infty n^2|a_n|^2\right)^{1/2}\left(\sum_{n=1}^\infty \frac{1}{n^2}\right)^{1/2}\right]^2\\
&= \left(\sum_{n=1}^\infty\frac{1}{n^2}\right)\left[(|a_0|^2)^{1/2} + \left(\sum_{n=1}^\infty n^2|a_n|^2\right)^{1/2}\right]^2\\
&\leq 2\left(\sum_{n=1}^\infty\frac{1}{n^2}\right)\left(|a_0|^2+\sum_{n=1}^\infty n^2|a_n|^2\right)\\
&= \displaystyle\frac{\pi^2}{3}\|f\|_{S^2}^2.
\end{aligned}$$
Thus $\|f\|_\infty \leq \displaystyle\frac{\pi\sqrt{3}}{3}\|f\|_{S^2}.$
\end{proof}

This result not only shows that $S^2$ is contained in the disk algebra $A(\D)$, but also the inclusion map is continuous.  Although this is a known result (see \cite{Duren:70} or \cite{ContrerasHernandez-Diaz:04}), we establish an actual bound which will be used in developing the estimates on the norm of weighted composition operators on $S^2$.

\begin{lemma}\label{Heller-Lemma} Let $\varphi$ be an automorphism of $\D$ and $f,g \in S^2$.  Then
\begin{enumerate}
\item[(a)] $|f(\varphi(0))|^2 \leq \|K_{\varphi(0)}\|_{S^2}^2\|f\|_{S^2}^2 \leq \displaystyle\frac{\|f\|_{S^2}^2}{1-|\varphi(0)|^2}.$
\item[(b)] $\|(f\circ\varphi)'\|_{H^2}^2 \leq \left(\displaystyle\frac{1+|\varphi(0)|}{1-|\varphi(0)|}\right)\|f\|_{S^2}^2.$
\item[(c)] $fg \in S^2$ and $\|(fg)'\|_{H^2} \leq \displaystyle\frac{2\pi\sqrt{3}}{3}\|f\|_{S^2}\|g\|_{S^2}$
\end{enumerate}
\end{lemma}

\begin{proof}
The proof of parts (a) and (b) can be found as part of the proof of Theorem 3.2 in \cite{Heller:10}.  We will provide a proof of part (c).  Since both $f$ and $g$ are assumed to be in $S^2$, they are necessarily bounded on $\D$ by Theorem \ref{properties}(b).  By the triangle inequality and Proposition \ref{sup-norm-inequality}, we have
$$\begin{aligned}
\|(fg)'\|_{H^2} &= \|f'g + fg'\|_{H^2}\\
&\leq \|f'g\|_{H^2} + \|fg'\|_{H^2}\\
&\leq \|f'\|_{H^2}\|g\|_\infty + \|f\|_\infty\|g'\|_{H^2}\\
&\leq \|f\|_{S^2}\|g\|_\infty + \|f\|_\infty\|g\|_{S^2}\\
&\leq \|f\|_{S^2}\left(\displaystyle\frac{\pi\sqrt{3}}{3}\|g\|_{S^2}\right) + \left(\displaystyle\frac{\pi\sqrt{3}}{3}\|f\|_{S^2}\right)\|g\|_{S^2}\\
&= \frac{2\pi\sqrt{3}}{3}\|f\|_{S^2}\|g\|_{S^2},
\end{aligned}$$ as desired.
\end{proof}

\begin{proposition}\label{factorization} If $f \in S^2$ and $w \in \D$ is a zero of $f$ of order $n \in \N$, then there exists a function $g \in S^2$ such that $g(w) \neq 0$ and $f(z) = (z-w)^ng(z)$ for all $z \in \D$.
\end{proposition}

\begin{proof}
Since $f$ is analytic on $\D$, there exists a function $g$ analytic on $\D$ with $g(w) \neq 0$ and $f(z) = (z-w)^ng(z)$.  It suffices to show that $g(z) = f(z)/(z-w)^n$ is in $S^2$.  Let $f$ have a power series representation $f(z) = \sum_{n=0}^\infty a_nz^n$ for all $z \in \D$.  Then $$g(0) = \begin{cases} a_{n+1}, & w = 0\\\displaystyle\frac{f(0)}{(-w)^n}, & w\neq 0.\end{cases}$$  

We will now show $\|g'\|_{H^2} < \infty$.  Let $\alpha = \frac{1}{2}(1+|w|)$.  Then for $\alpha < r < 1$ and $|z| = r$, we have $$\frac{1}{|z-w|} \leq \frac{2}{1-|w|}.$$  By the previous estimate, the triangle inequality, Theorem 1.5 of \cite{Duren:70} and inequality (\ref{pons-inequality}) we obtain
$$\begin{aligned}
\|g'\|_{H^2}^2 &= \sup_{\alpha < r < 1} \int_0^{2\pi} \left|g'(re^{i\theta})\right|^2\;\frac{d\theta}{2\pi}\\
&\leq \frac{2^{2n+3}}{\pi(1-|w|)^{2n+2}}\sup_{\alpha < r < 1} \int_0^{2\pi} \left(4|f'(re^{i\theta})|^2 + n^2|f(re^{i\theta})|^2\right)\;d\theta,
\end{aligned}$$
which is finite since $f \in S^2 \subset H^2$.  Thus $g \in S^2$ as desired.
\end{proof}

For $w \in \D$, we define $ev_{w}$ to be the linear functional on $S^2$ such that $ev_{w}(f) = f(w)$ for $f \in S^2$.  Since $S^2$ is a functional Hilbert space, $ev_{w}$ is a bounded linear functional for every $w \in \D$.

\begin{proposition} Let $w \in \D$.  Then $\ker ev_{w} = \mathrm{ran }(M_z-w)$.
\end{proposition}

\begin{proof} Let $g \in \mathrm{ran }(M_z - w)$.  Then there exists an $f \in S^2$ such that $(z-w)f(z) = g(z)$ for $z \in \D$.  Then $0 = g(w) = ev_w(g)$, and so $\mathrm{ran }(M_z - w) \subseteq \ker ev_w$.

Now suppose $f \in \ker ev_w$ with $f$ not identically zero.  By Proposition \ref{factorization}, there exists $g \in S^2$ with $g(w) \neq 0$ and $f(z) = (z-w)^ng(z)$.  It follows that $G(z) = (z-w)^{n-1}g(z)$ is in $S^2$ by Lemma \ref{Heller-Lemma}(c).  Thus $f(z) = (z-w)G(z)$, and hence $\ker ev_w \subseteq \mathrm{ran }(M_z-w)$.
\end{proof}

\section{Boundedness and Norm Estimates}\label{section:boundedness}
In this section we establish a characterization of the bounded multiplication operators on $S^2$, as well as estimates on the operator norm.  Our aim is to characterize the boundedness of $M_\psi$ in terms of the properties of $\psi$.  To accomplish this, we first investigate the weighted composition operators on $S^2$ induced by automorphisms of $\D$.

Contreras and Hern\'{a}ndez-D\'{i}az in \cite{ContrerasHernandez-Diaz:04} characterized bounded weighted composition operators on $S^2$, among other things, in terms of the boundedness of a different weighted composition operator on $H^2$.  In this paper, we establish the boundedness of the weighted composition operator on $S^2$ in terms of the symbols.  Specifically, we study weighted composition operators $W_{\psi,\varphi}$ induced by $\psi \in S^2$ and $\varphi$ an automorphism of $\D$.  With this, we can establish the boundedness of $M_\psi$ on $S^2$, as well as operator norm estimates on $W_{\psi,\varphi}$ and $M_\psi$.

\begin{theorem}\label{WCO-bounded} Let $\psi$ be analytic on $\D$ and $\varphi$ an automorphism of $\D$.  Then $W_{\psi,\varphi}$ is bounded on $S^2$ if and only if $\psi \in S^2$.  Moreover, for $W_{\psi,\varphi}$ bounded on $S^2$, $$\|\psi\|_{S^2} \leq \|W_{\psi,\varphi}\| \leq \left[\left(1+\frac{8\pi^2}{3}\right)\left(\frac{1+|\varphi(0)|}{1-|\varphi(0)|}\right)\right]^{1/2}\|\psi\|_{S^2}.$$
\end{theorem}

\begin{proof}
Suppose $W_{\psi,\varphi}$ is bounded on $S^2$.  Taking the test function $f$ identically 1, we have
$\|\psi\|_{S^2} = \|W_{\psi,\varphi} f\|_{S^2} < \infty.$  Thus $\psi \in S^2$.

Now suppose $\psi \in S^2$.  For $f \in S^2$, we have $f\circ\varphi \in S^2$.  By Lemma \ref{Heller-Lemma}, we obtain
$$\begin{aligned}
\|W_{\psi,\varphi} f\|_{S^2}^2 &= |\psi(0)f(\varphi(0))|^2 + \|(\psi(f\circ\varphi))'\|_{H^2}^2\\
&\leq \|\psi\|_{S^2}^2|f(\varphi(0))|^2 + \frac{4\pi^2}{3}\|\psi\|_{S^2}^2\|f\circ\varphi\|_{S^2}^2\\
&\leq \left[\left(1+\frac{4\pi^2}{3}\right)|f(\varphi(0))|^2+\frac{4\pi^2}{3}\|(f\circ\varphi)'\|_{H^2}^2\right]\|\psi\|_{S^2}^2\\
&\leq \left[\left(1+\frac{4\pi^2}{3}\right)\frac{1}{1-|\varphi(0)|^2} + \frac{4\pi^2}{3}\left(\frac{1+|\varphi(0)|}{1-|\varphi(0)|}\right)\right]\|\psi\|_{S^2}^2\|f\|_{S^2}^2\\
&\leq \left(1+\frac{8\pi^2}{3}\right)\left(\frac{1+|\varphi(0)|}{1-|\varphi(0)|}\right)\|\psi\|_{S^2}^2\|f\|_{S^2}^2.
\end{aligned}$$
Thus $W_{\psi,\varphi}$ is bounded on $S^2$ and the upper estimate follows.  To obtain the lower bound, we take the function $f$ identically 1 and observe $\|\psi\|_{S^2} = \|W_{\psi,\varphi} f\|_{S^2}$.
\end{proof}

As a direct consequence we obtain a characterization of the multiplication operators on $S^2$ in terms of the symbol $\psi$ by taking $\varphi(z) = z$.

\begin{corollary}\label{bounded-characterization} The multiplication operator $M_\psi$ is bounded on $S^2$ if and only if $\psi \in S^2$.
\end{corollary}

We now establish bounds on the norm of $M_\psi$ on $S^2$.  From Theorem \ref{WCO-bounded}, we can conclude that $$\|\psi\|_{S^2} \leq \|M_\psi\| \leq \left(1+\displaystyle\frac{8\pi^2}{3}\right)^{1/2}\|\psi\|_{S^2}.$$  However, many of the estimates used in the proof are not sharp.  This is due, in part, to the composition operator.  While this is not a concern when determining boundedness of an operator, it poses problems when determining norm estimates.

When considering the multiplication operator in isolation, we obtain sharper estimates on the norm.

\begin{theorem}\label{mop-norm}
Let $\psi \in S^2$.  Then $$\max\left\{\|\psi\|_{S^2},\|\psi\|_\infty\right\} \leq \|M_\psi\| \leq \left(1+\frac{4\pi^2}{3}\right)^{1/2}\|\psi\|_{S^2}.$$
\end{theorem}

\begin{proof}
From Lemma 11 of \cite{DurenRombergShields:69}, we have $\|\psi\|_\infty \leq \|M_\psi\|$.  Thus, the lower bound is established.

By Lemma \ref{Heller-Lemma}, for $f \in S^2$
$$\begin{aligned}
\|M_\psi f\|_{S^2}^2 &= |\psi(0)f(0)|^2 + \|(\psi f)'\|_{H^2}^2\\
&\leq \|\psi\|_{S^2}^2\|f\|_{S^2}^2 + \displaystyle\frac{4\pi^2}{3}\|\psi\|_{S^2}^2\|f\|_{S^2}^2\\
&= \left(1+\frac{4\pi^2}{3}\right)\|\psi\|_{S^2}^2\|f\|_{S^2}^2.
\end{aligned}$$  Thus $$\|M_\psi f\|_{S^2} \leq \left(1+\frac{4\pi^2}{3}\right)^{1/2}\|\psi\|_{S^2}\|f\|_{S^2},$$ and the upper estimate follows.
\end{proof}

\section{Spectrum}\label{section:spectrum}
We now turn our attention to the spectrum of the bounded multiplication operators on $S^2$.  Recall that the \textit{resolvent set} of a bounded linear operator $T$ on a complex Banach space $\mathcal{X}$ is defined as
$$\rho(T) = \{\lambda \in \C : T-\lambda I \text{ is invertible}\},$$ where $I$ is the identity operator.  The \textit{spectrum} of $T$ is defined as $\sigma(T) = \C\setminus\rho(T)$.

The spectrum is a non-empty compact subset of the closed disk centered at the origin of radius $\|T\|$.  The set of \textit{eigenvalues} of $T$, also known as the \textit{point spectrum}, is a subset of the spectrum and defined as $$\sigma_p(T) = \{\lambda \in \C : \ker(T-\lambda I) \neq \{0\}\}.$$

\begin{theorem}\label{no-eigenvalues} Let $\psi \in S^2$.  Then $M_\psi$ has no eigenvalues unless $\psi$ is a constant function, that is,
$$\sigma_p(M_\psi) = \begin{cases}
\{\lambda\}, &\text{ if $\psi(z) = \lambda$ for all $z \in \D$};\\
\emptyset, &\text{ otherwise}.
\end{cases}$$
\end{theorem}

\begin{proof}
Suppose $\psi(z) = \lambda$ for all $z \in \D$.  Then for any non-zero function $f \in S^2$, it is true that $M_\psi f - \lambda f$ is identically zero on $\D$.  Thus $\lambda$ is an eigenvalue of $M_\psi$.  Let $\mu \in \C$ such that $\mu \neq \lambda$.  If $f \in S^2$ such that $M_\psi f - \mu f$ is identically zero on $\D$, then $f$ is identically zero.  Thus $\mu$ is not an eigenvalue of $M_\psi$.  So $\sigma_p(M_\psi) = \{\lambda\}$.

Now suppose $\psi$ is not a constant function.  Assume $\lambda \in \C$ is an eigenvalue of $M_\psi$ and $f \in S^2$ is a corresponding (non-zero) eigenfunction.  Since $M_\psi f - \lambda f$ is identically zero on $\D$, it must be the case that $f$ is identically zero (see Chapter IV Theorem 3.7 of \cite{Conway:78}), a contradiction.  Thus $M_\psi$ has no eigenvalues.
\end{proof}

\begin{remark} \label{spec_remark} If the multiplication operator $M_\psi$ is bounded on some Banach space of analytic functions, then $\lambda \in \sigma(M_\psi)$ if and only if $M_\psi - \lambda I$ is not invertible.  Since $\lambda I = M_\lambda$, and $M_\psi - M_\lambda = M_{\psi-\lambda}$, we see that $\lambda \in \sigma(M_\psi)$ if and only if $M_{\psi-\lambda}$ is not invertible.  Since $M_{\psi-\lambda}^{-1} = M_{(\psi-\lambda)^{-1}}$, we finally arrive at $\lambda \in \sigma(M_\psi)$ if and only if $M_{(\psi-\lambda)^{-1}}$ is not well-defined.
\end{remark}

\begin{theorem}\label{spectrum}
Let $\psi \in S^2$.  Then $\sigma(M_\psi) = \psi\left(\overline{\D}\right)$.
\end{theorem}

\begin{proof}
Suppose $\lambda \not\in \sigma(M_\psi)$.  Then $M_\psi - \lambda I$ is onto.  So there exists a function $f \in S^2$ such that $(\psi(z) - \lambda)f(z) = 1$ for all $z \in \D$.  It follows that $\lambda \not\in \psi(\D)$.  Thus $\psi(\D) \subseteq \sigma(M_\psi)$.

Now suppose $\lambda \not\in \psi\left(\overline{\D}\right)$.  Then the function $\psi-\lambda$ is bounded away from zero, that is there exists $c > 0$ such that $|\psi(z)-\lambda| \geq c$ for all $z \in \D$.  Thus the function $g = (\psi-\lambda)^{-1}$ is a bounded analytic function on $\D$.  We see that $g \in S^2$ since
$$\begin{aligned}\|g\|_{S^2}^2 &= |g(0)|^2 + \|g'\|_{H^2}^2\\
&= \frac{1}{|\psi(0)-\lambda|^2} + \left\|\frac{\psi'}{(\psi-\lambda)^2}\right\|_{H^2}^2\\
&\leq \frac{1}{c^2} + \frac{1}{c^4}\|\psi'\|_{H^2}^2\\
&\leq \frac{1}{c^2} + \frac{1}{c^4}\|\psi\|_{S^2}^2 < \infty.
\end{aligned}$$  Thus $M_g$ is bounded on $S^2$ by Corollary \ref{bounded-characterization}, and so $\lambda \not\in \sigma(M_\psi)$ (see Remark \ref{spec_remark}).  Since the spectrum is closed, we have $\psi\left(\overline{\D}\right) = \sigma(M_\psi)$.
\end{proof}

\section{Compactness}\label{section:compact}
In this section, we characterize the compact multiplication operators on $S^2$.  As with many spaces, the only compact multiplication operator is that induced by the zero function.  To prove this, we utilize the spectral theorem for compact operators due to F. Riesz (see Chapter VII Theorem 7.1 of \cite{Conway:90}), and the results from the previous section.

\begin{theorem} Let $\psi \in S^2$.  Then $M_\psi$ is compact on $S^2$ if and only if $\psi$ is identically zero.
\end{theorem}

\begin{proof}
Clearly if $\psi$ is identically zero, then $M_\psi$ is compact on $S^2$.  Suppose $M_\psi$ is compact on $S^2$.  The spectrum of $M_\psi$ is connected since it is precisely $\psi(\overline{\D})$ by Theorem \ref{spectrum}.  Since $\sigma(M_\psi)$ is countable and contains 0, by the Spectral Theorem for Compact Operators, it must be the case that $\psi(\overline{\D}) = \sigma(M_\psi) = \{0\}$.  Thus $\psi$ is identically zero. 
\end{proof}

This result is also a consequence of Theorem 2.1 of \cite{ContrerasHernandez-Diaz:04}.  However, our proof is independent of the theory of compact multiplication operators on $H^2$.

\section{Isometries}\label{section:isometries}
In this section, we characterize the isometric multiplication operators, as well as the isometric zero-divisors on $S^2$.

\begin{theorem}\label{isometries} Let $\psi \in S^2$.  Then $M_\psi$ is an isometry on $S^2$ if and only if $\psi$ is a constant function of modulus one.
\end{theorem}

\begin{proof}
It is clear that a constant function of modulus one induces an isometric multiplication operator on $S^2$.  Now suppose $\psi \in S^2$ such that $M_\psi$ is an isometry on $S^2$.  For $f$ identically 1, we have $$\|\psi\|_{S^2} = \|M_\psi f\|_{S^2} = \|f\|_{S^2} = 1.$$  Thus, if $\psi$ has power series representation $\psi(z) = \sum_{n=0}^\infty a_nz^n$,
then
\begin{equation}\label{series-norm}
1 = \|\psi\|_{S^2}^2 = |a_0|^2 + \sum_{n=1}^\infty n^2|a_n|^2,
\end{equation}
and
$$z\psi'(z) + \psi(z) = a_0 + \sum_{n=1}^\infty (n+1)a_nz^n.$$

\noindent Let $g(z) = z$.  Since $M_\psi$ is an isometry on $S^2$, we have
$$\begin{aligned}
1 &= \|g\|_{S^2}^2 = \|M_\psi g\|_{S^2}^2 = \|z\psi' + \psi\|_{H^2}^2\\
&= |a_0|^2 + \sum_{n=1}^\infty(n+1)^2|a_n|^2\\
&= |a_0|^2 + \sum_{n=1}^\infty n^2|a_n|^2 + 2\sum_{n=1}^\infty n|a_n|^2 + \sum_{n=1}^\infty |a_n|^2\\
&= \|\psi\|_{S^2}^2 + 2\sum_{n=1}^\infty n|a_n|^2 + \sum_{n=1}^\infty |a_n|^2.
\end{aligned}$$  From (\ref{series-norm}), it follows immediately that $$\sum_{n=1}^\infty |a_n|^2 = 0,$$ and thus $a_n = 0$ for all $n \geq 1$.  So $\psi$ is the constant function $\psi(z) = a_0$, and by (\ref{series-norm}) again, $|\psi(z)| = |a_0| = 1$, for all $z \in \D$.
\end{proof}

The following discussions are consequences of the characterization of the isometric multiplication operators on $S^2$.

\subsection{Isometric Zero-Divisors}
A nonempty sequence (finite or infinite) $\{z_k\}$ of $\D$ is called a \textit{zero-set} of $S^2$ if there exists a function $f \in S^2$ which vanishes precisely on $\{z_k\}$.  For such a zero-set, a function $g \in S^2$ is called a \textit{zero-divisor} if it vanishes precisely on $\{z_k\}$ and $f/g \in S^2$ for all $f \in S^2$ which vanish on $\{z_k\}$.  If $\|f/g\|_{S^2} = \|f\|_{S^2}$ for every such function $f$, then $g$ is said to be an \textit{isometric zero-divisor}.

Aleman, et. al. proved the following necessary condition for the existence of isometric zero-divisors for a class of functional Banach spaces which includes $S^2$.

\begin{theorem}[Theorem 5 of \cite{AlemanDurenMartinVukotic:10}]\label{divisors} Let $\mathcal{X}$ be a functional Banach space, and suppose $\mathcal{X}$ is invariant under multiplication by polynomials.  Then every isometric zero-divisor in $\mathcal{X}$ induces an isometric multiplication operator on $\mathcal{X}$.
\end{theorem}

\noindent Due to the characterization of isometric multiplication operators on $S^2$, we obtain the following.

\begin{theorem} There are no isometric zero-divisors in $S^2$. \end{theorem}

\begin{proof}
It has already been shown that $S^2$ is a functional Hilbert space.  Since the polynomials are contained in $S^2$, Lemma \ref{Heller-Lemma}(c) implies that $S^2$ is invariant under multiplication by polynomials.  Thus by Theorem \ref{divisors}, if $g$ is an isometric zero-divisor of $S^2$, then $g$ induces an isometric multiplication operator on $S^2$.  However, by Theorem \ref{isometries} it must be the case that $g$ is a constant function of modulus one, which does not vanish anywhere.  Thus $S^2$ contains no isometric zero-divisors.
\end{proof}

\bibliographystyle{amsplain}
\bibliography{references.bib}
\end{document}